\documentclass[reqno]{amsart}

\usepackage{amsmath,amsthm}
\usepackage{a4wide}

\usepackage{graphicx}

\usepackage{txfonts}
\renewcommand{\leq}{\leqslant}
\renewcommand{\geq}{\geqslant}

\usepackage[numbers]{natbib}

\usepackage{graphicx}
\usepackage{caption}
\usepackage[labelformat=simple]{subcaption}

\usepackage{xcolor}

\usepackage{hyperref}
\hypersetup{
colorlinks,
citecolor=violet,
linkcolor=red,
urlcolor=blue
}

\usepackage{enumitem}
\usepackage[capitalise]{cleveref}

\setlist[enumerate,1]{label=\roman*),ref=\roman*)}
\setlist[enumerate,2]{label=\alph*),ref=\roman{enumi}-\alph*)}
\setlist[enumerate,3]{label=(\Alph*),ref=(\roman{enumi}-\alph{enumii}-\Alph*)}
\setlist[enumerate,4]{label=(\arabic*),ref=(\roman{enumi}-\alph{enumii}-\Alph{enumiii}-\arabic*)}

\theoremstyle{plain}
\newtheorem{theorem}{Theorem}[section]
\newtheorem{proposition}[theorem]{Proposition}
\newtheorem{lemma}[theorem]{Lemma}
\newtheorem{corollary}[theorem]{Corollary}

\theoremstyle{definition}

\newtheorem{problem}{Problem}
\numberwithin{equation}{section}

\newcommand{\pone}{p_1}
\newcommand{\ptwo}{p_2}

\newcommand{\R}{\mathbb R}

\newcommand{\ee}{\mathrm{e}}
\newcommand{\dd}{\mathrm{d}}

\newcommand{\Ber}{\mathrm{Ber}}
\newcommand{\PoiBin}{\mathrm{PoiBin}}

\newcommand{\E}{\mathbf{E}}
\newcommand{\Prb}{\mathbf{P}}

\renewcommand{\vec}{\mathbf}

\title{The Kearns--Saul inequality for Bernoulli and Poisson-binomial distributions}%
\author{Eckhard Schlemm}%
\address{Wolfson College, University of Cambridge}
\curraddr{UCL Medical School, University College London}
\email{eckhard.schlemm@cantab.net}


\begin{document}

\begin{abstract}
We give a direct rigorous proof of the Kearns--Saul inequality which bounds the Laplace transform of a generalised Bernoulli random variable. We extend the arguments to generalised Poisson-binomial distributions and characterise the set of parameters such that an analogous inequality holds for the sum of two generalised Bernoulli random variables.
\end{abstract}

\maketitle

\section{Introduction and main results}
A generalised Bernoulli random variable $X\sim \Ber(p)$ with parameter $p\in[0,1]$ is defined by its distribution function $\Prb(X=1-p)=1-\Prb(X=-p)=p$. It differs from a classical Bernoulli random variable in that it is shifted so as to have mean zero. In \cite{kearns1998large} the Laplace transform $\E\ee^{tX}$ of a generalised  Bernoulli random variable $X\sim\Ber(p)$ was bounded by
\begin{equation}
\label{eq-KSineq}
p\ee^{t(1-p)}+(1-p)\ee^{-tp}\leq\exp\left[\frac{1-2p}{4\log[(1-p)/p]}t^2\right],\quad t\in\R,\quad 0\leq p\leq 1.
\end{equation}
A rigorous proof of this inequality was provided in \cite{berend2012concentration}, where the function
\begin{equation}
\label{gorig}
g_p(t) = \frac{1}{t^2}\log\left[p\ee^{t(1-p)}+(1-p)\ee^{-tp}\right],\quad t\in\R,
\end{equation}
was analysed using convexity arguments. In the same paper, the task of proving that the function $g_p$ is strictly unimodal with a unique maximum at $t=t_p^{\ast}= 2\log[(1-p)/p]$ was classified as an "intriguing open problem". Here, a differentiable real-valued function $f$ on $\R$ is said to be strictly unimodal if there exists an $x$ such that the derivative $f'$ is positive on $(-\infty,x)$ and negative on $(x,\infty)$. In the next section we provide a proof of the following solution to this problem.
\begin{theorem}
\label{theorem-KS-orig}
For every $p\in[0,1]$, the function $g_p$ defined in \cref{gorig} is strictly unimodal.
\end{theorem}

A natural extension of generalised Bernoulli random variables is the family of Poisson-binomial distributions \cite{chen1997statistical}. For a positive integer $n$ and a parameter vector $\vec p=(\pone,\ldots,p_n)\in[0,1]^n$, the $\PoiBin_n(\vec p)$ distribution is defined as the distribution of the random variable $X_1+\ldots+X_n$, where the $X_i\sim \Ber(p_i)$ are independent generalised Bernoulli random variables. In the following we will be interested in the case $n=2$ and provide a generalisation of \cref{eq-KSineq}. The statement and proof of this generalisation, as well as \cref{cor-KSInew}, are the main results of this paper. The analogue of the function $g_p$, defined in \cref{gorig}, which occupies a central role in the proof of the Kearns--Saul inequality, is
\begin{equation}
\label{gnew}
g_{\pone,\ptwo}(t) = g_{\pone}(t)+g_{\ptwo}(t) = \frac{1}{t^2}\log\left[\ee^{-t (\pone +\ptwo)}\left(1+ \pone \left(\ee^t-1\right) \right) \left(1+\ptwo \left(\ee^t-1\right) \right)\right].
\end{equation}
By $\langle x,y\rangle$ we denote an ordered pair of real numbers $x$ and $y$.
\begin{theorem}
\label{thm-unimodalitynew}
Let $g_{\pone,\ptwo}$ be the function defined in \cref{gnew} and 
\begin{equation}
t_{\pone,\ptwo}^{\ast}=\log\left[\frac{1-\pone}{\pone}\frac{1-\ptwo}{\ptwo}\right].
\end{equation}
Then $g_{\pone,\ptwo}\left(t_{\pone,\ptwo}^{\ast}\right)=0$ and there exist sets $C\subset B\subset A\subset [0,1]^2$ such that
\begin{enumerate}
\item\label{thm-unimodalitynew1} $g_{\pone,\ptwo}'(t)$ is positive for $t<0$ if and only if $\langle\pone,\ptwo\rangle\in A$;
\item\label{thm-unimodalitynew2} $g_{\pone,\ptwo}'(t)$ is negative for $t>t_{\pone,\ptwo}^{\ast}$ if and only if $\langle\pone,\ptwo\rangle\in B$;
\item\label{thm-unimodalitynew3} $g_{\pone,\ptwo}'(t)$ is positive for $0<t<t_{\pone,\ptwo}^{\ast}$ if and only if $\langle\pone,\ptwo\rangle\in C$.
\end{enumerate}
In particular, $g_{\pone,\ptwo}$ is unimodal if and only if $\langle\pone,\ptwo\rangle\in C$.
\end{theorem}
The closed, convex sets $A$ and $B$ are given explicitly by \cref{eq-inclusionDA,eq-defB}. The boundary of the set $C$ can be determined numerically by solving the differential equation \labelcref{eq-secondordercondition}. The following result gives an explicit sufficient condition for $g_{\pone,\ptwo}$ to be unimodal.
\begin{theorem}
\label{thm-sufficientunimodality}
A sufficient condition for $g_{\pone,\ptwo}$ to be unimodal is $\langle\pone,\ptwo\rangle\in D$, where
\begin{equation}
D=\left\{\langle\pone,\ptwo\rangle\in[0,1]^2:12\pone \ptwo \left(\pone \ptwo  + \pone  + \ptwo \right) - 14\pone \ptwo  + \pone ^2+\ptwo ^2\leq 0\right\}
\end{equation}
is a closed convex subset of $C$.
\end{theorem}
The sets $A$, $B$, $C$ and $D$ are depicted in \cref{fig-g-regions}. One can see that $B\backslash D$ is fairly small and that the inclusion $D\subset C\subset B$ thus contains a considerable amount of information about the shape of $C$. As a corollary to \cref{thm-unimodalitynew} we obtain the following result, a direct generalisation of the Kearns--Saul inequality for Poisson-binomial random variables.
\begin{figure}
\centering
\includegraphics[width=0.75\textwidth]{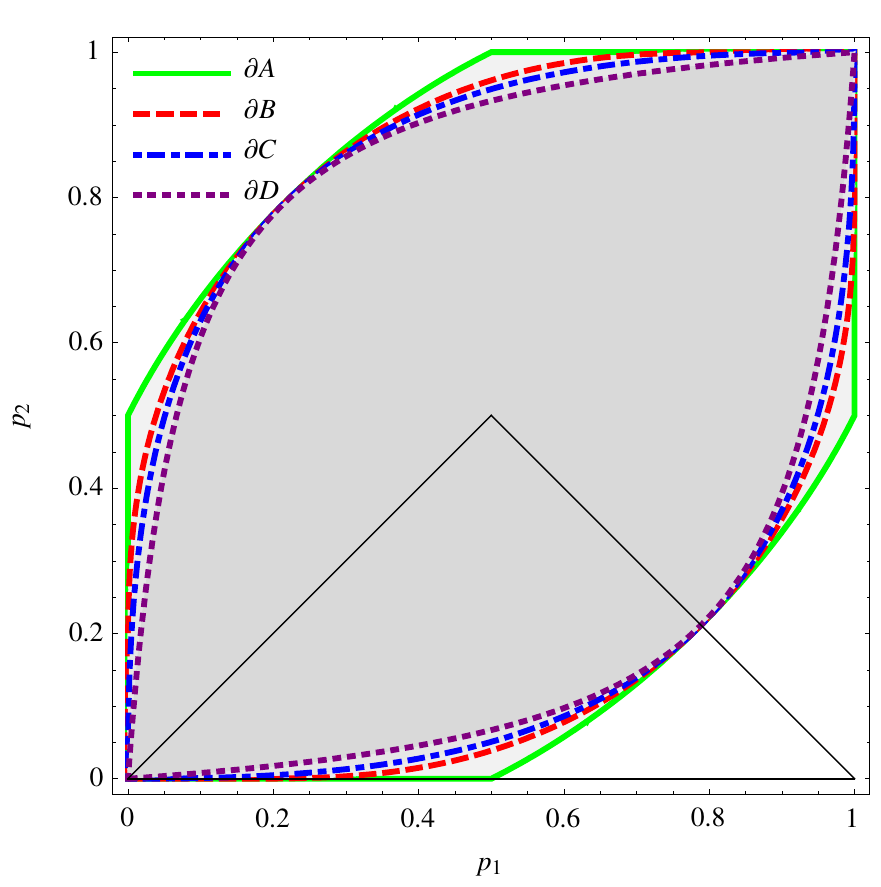}
\caption{Illustration of the sets appearing in \cref{thm-unimodalitynew} about the behaviour of the function $g_{\pone,\ptwo}$ (defined in \cref{gnew}) in different regions of the parameter space $[0,1]^2$. For $\langle\pone,\ptwo\rangle\in A$, the region with the green solid boundary, the derivative $g_{\pone,\ptwo}'$ is positive on the interval $\left(-\infty,0\right)$. For $\langle\pone,\ptwo\rangle\in B$, the region with the red dot-dashed boundary, the derivative $g_{\pone,\ptwo}'$ is negative on the interval $\left(t_{\pone,\ptwo}^{\ast},\infty\right)$. For $\langle\pone,\ptwo\rangle\in C$, the region with the blue dash-dotted boundary, the derivative $g_{\pone,\ptwo}'$ is positive on the interval $\left(0,t_{\pone,\ptwo}^{\ast}\right)$. For $\langle\pone,\ptwo\rangle\in D$, the region with the purple dotted boundary, the function $g_{\pone,\ptwo}$ has only two inflection points and is unimodal by \cref{thm-sufficientunimodality}. The triangular region $\Delta$, to which attention is restricted in the proof of \cref{thm-unimodalitynew}, is bounded by black lines.}
\label{fig-g-regions}
\end{figure}
\begin{corollary}
\label{cor-KSInew}
For every $\langle\pone,\ptwo\rangle\in C\subset[0,1]^2$, the Laplace transform of the generalised Poisson-binomial random variable $X\sim\PoiBin_2(\pone,\ptwo)$ satisfies
\begin{equation}
\label{eq-KSineq-new}
\E\ee^{tX}\leq \exp\left[\frac{1-\pone -\ptwo}{\log[(1-\pone)/\pone (1-\ptwo)/\ptwo ]}t^2\right],\quad t\in\R.
\end{equation}
\end{corollary}
\begin{proof}
The distribution function of a $\PoiBin_2(\pone,\ptwo)$-distributed random variable $X$ is given by
\begin{align*}
\Prb(X=-\pone -\ptwo)=&(1-\pone)(1-\ptwo),\\
\Prb(X=1-\pone -\ptwo)=&\pone (1-\ptwo)+(1-\pone)\ptwo,\\
\Prb(X=2-\pone -\ptwo)=&\pone \ptwo,
\end{align*}
and its Laplace transform therefore takes the form
\begin{align*}
\E\ee^{tX} =& (1-\pone)(1-\ptwo)\ee^{-t(\pone +\ptwo)} + (\pone +\ptwo -2\pone \ptwo)\ee^{t(1-\pone -\ptwo)} + \pone \ptwo \ee^{t(2-\pone -\ptwo)}\\
    =&\ee^{-t(\pone +\ptwo)} \left[1+\pone \left(e^t-1\right)\right] \left[1+\ptwo \left(e^t-1\right)\right]\\
    =& \exp\left[t^2 g_{\pone,\ptwo}(t)\right].
\end{align*}
\Cref{thm-unimodalitynew} implies that, for $\langle\pone,\ptwo\rangle\in C$, the function $g_{\pone,\ptwo}$ has a unique maximum at $t=t_{\pone,\ptwo}^{\ast}$, and therefore satisfies
\begin{equation}
g_{\pone,\ptwo}(t)\leq g_{\pone,\ptwo}\left(t_{\pone,\ptwo}^{\ast}\right) = \frac{1-\pone -\ptwo}{\log[(1-\pone)/\pone (1-\ptwo)/\ptwo ]},\quad t\in\R.
\end{equation}
The claim follows.
\end{proof}
We conclude this section with a some remarks and open problems. We have seen that the sets $A$, $B$ and $D$ are convex and from \cref{fig-g-regions} it appears that the same true for the set $C$.
\begin{problem}
Show that the set $C$ appearing in the statement of \cref{thm-unimodalitynew} is convex.
\end{problem}
Based on numerical computations, we conjecture that inequality \labelcref{eq-KSineq-new} does, in fact, hold for a parameter region $C'$ slightly larger than $C$, even though $g_{\pone,\ptwo}$ would not be unimodal for $\langle\pone,\ptwo\rangle\in C'\backslash C$.
\begin{problem}
Characterise the set of parameters $\langle\pone,\ptwo\rangle$ such that inequality \labelcref{eq-KSineq-new} holds.
\end{problem}
An extension of our results to $\PoiBin_n$-distributions with $n$ larger than two appears to be very difficult. One reason is that, as $n$ increases, the number of critical points of the analogue of the function $g_{\pone,\ptwo}$ increases and we do not know, in general, how to find the abscissa of the critical point corresponding to the global maximum.

\section{Proof for generalised Bernoulli random variables}
In this section we prove \cref{theorem-KS-orig} about the unimodality of the function $g_p$, defined in \cref{gorig}. In the following, by unimodality, concavity and convexity, we always mean strict unimodality, strict concavity and strict convexity. The derivative of the function $g_p$ is given by
\begin{align*}
g_p'(t)=\frac{(1-p)pt\left(\ee^t-1\right) - 2\left(1+p\left(\ee^t-1\right)\right) \log\left[\ee^{-p t}\left(1+p\left(\ee^t-1\right)\right)\right]}{\left(1+p\left(\ee^t-1\right)\right) t^3},
\end{align*}
and $g_p'(t)$ vanishes for $t=t_p^{\ast}=2\log[(1-p)/p]$. In order to prove unimodality we will, without loss of generality, assume that $p$ is less than $1/2$ and $t_p^{\ast}$ thus positive. If $p>1/2$, one may consider instead the random variable $-X$, which is $\Ber(1-p)$-distributed. The boundary case $p=1/2$ is best dealt with separately: in this case the function $g_{1/2}$ is symmetrical with a maximum at $t=0$, and is easily seen to be unimodal. We first record the following easy properties of the function $f_p:t\mapsto t^3 g_p'(t)$ for later reference.
\begin{lemma}
\label{lemma-inflectionpoints-old}
For every $p<1/2$, the function $f_p$ has exactly two inflection points; their  abscissas are $t=0$ and $t=t_p^{\ast}/2=\log[(1-p)/p]$. The function $f_p$ is concave on $\R^-\cup \left(t_p^{\ast}/2,\infty\right)$ and convex on $\left(0,t_p^{\ast}/2\right)$.
\end{lemma}
\begin{proof}
To prove the lemma, we first compute the second derivative of $f_p$ which equals
\begin{equation*}
 f_p''(t) = \frac{p(1-p)t\ee^t(1-p-p\ee^t)}{\left(1+p\left(\ee^t-1\right)\right)^3}.
\end{equation*}
This expression vanishes exactly for $t=0$ and $t=t_p^{\ast}/2$. To conclude that these are indeed the abscissas of reflection points we need to verify that the third derivative of $f_p$ does not vanish there. We find that $f_p^{(3)}(0)=p(1-p)(1-2p)$ and $f_p^{(3)}\left(t_p^{\ast}/2\right) = \log[p/(1-p)]/8$, which are manifestly positive and negative, respectively, for $p<1/2$. 
\end{proof}

We can now give a proof of our first main result.
\begin{proof}[Proof of \cref{theorem-KS-orig}]
We will show that $g_p'(t)$ is positive for $t<t_p^{\ast}$ and negative for $t>t_p^{\ast}$ by analysing the sign of $f_p(t)=t^3g_p'(t)$ in the three intervals $\left(-\infty,0\right)$, $\left(0,t_p^{\ast}\right)$ and $\left(t_p^{\ast},\infty\right)$. We first show that $f_p(t)<0$ for $t<0$. Since $f_p$ itself as well as its first derivative vanish at $t=0$, the claim follows from the concavity of $f_p$ on $\R^-$. We next show that $f_p(t)<0$ for $t>t_p^{\ast}$. Since $f_p\left(t_p^{\ast}\right)=0$ it suffices by the concavity of $f_p$ on $\left(t_p^{\ast},\infty\right)$ to show that
\begin{equation*}
\mathfrak{r}(p)\coloneqq f_p'\left(t_p^{\ast}\right)=2p - 1 + 2(1-p) p \log[(1-p)/p]
\end{equation*}
is negative for $p<1/2$. This follows from the observation that $\mathfrak{r}$ is monotonely increasing with derivative $\mathfrak{r}'(p)=2(1-2p)\log[(1-p)/p]>0$  and the fact that $\mathfrak{r}(1/2)=0$. Lastly, we verify that $f_p(t)>0$ for $0<t<t_p^{\ast}$. This is a direct consequence of the facts that $f_p(0)=f_p'(0)=f_p''(0)=0$, $f_p^{(3)}(0)>0$, $f_p\left(t_p^{\ast}\right)=0$, $f_p'\left(t_p^{\ast}\right)<0$, and that $f_p$ has exactly one inflection point in the interval $\left(0,t_p^{\ast}\right)$.
\end{proof}

\section{Proof in the Poisson-Binomial case}
In this section we extend the previous arguments to the case of Poisson-binomial random variables. The derivative of the function $g_{\pone,\ptwo}$ defined in \cref{gnew} is given by
\begin{align*}
g_{\pone,\ptwo}'(t)=& g_{\pone}'(t) + g_{\ptwo}'(t)\\
  =&\frac{1}{t^3}\left\{t\left[2+\pone +\ptwo -\frac{1-\pone}{1-\pone \left(1-\ee^t\right) }-\frac{1-\ptwo}{1- \ptwo \left(1-\ee^t\right)}\right]\right.\\
    &\left.\qquad-2\log\left[\left(1-\pone \left(1-\ee^t\right) \right)\left(1- \ptwo \left(1-\ee^t\right)\right)\right]\right\},
\end{align*}
and this expression vanishes for $t=t_{\pone,\ptwo}^{\ast}=\log((1-\pone)/\pone (1-\ptwo)/\ptwo)$. In order to prove unimodality we will, without loss of generality, assume that $\pone +\ptwo$ is less than or equal to one and $t^{\ast}$ thus non-negative, and that $\pone\geq\ptwo$. This can always be guaranteed by considering $-X_i$ instead of $X_i$ and/or renaming the variables $X_1$, $X_2$. We thus concentrate on the triangular region $\left\{\langle\pone,\ptwo\rangle:\ptwo\leq (\pone,1-\pone)^-\right\}$ in the parameter space $[0,1]^2$. Here and in the following, $(x,y)^-$ denotes the minimum of two real numbers $x$ and $y$. As before, the boundary cases $\ptwo=(\pone,1-\pone)^-$ are dealt with first. Instead of $g_{\pone,\ptwo}$ we will often analyse the function $f_{\pone,\ptwo}: t\mapsto t^3 g_{\pone,\ptwo}'(t)$ which is better behaved at $t=0$. We also introduce the notation $p^\pm=(3\pm\sqrt 3)/6$.
\begin{lemma}
\label{lemma-boundarycase}
The function $f_{\pone,(\pone,1-\pone)^-}$ has the following properties.
\begin{enumerate}
\item\label{lemma-boundarycase-1}If $\pone<1/2$, then $f_{\pone,\pone}$ has roots at $t=0$ and $t=t_{\pone,\pone}^*$, and is negative on $\left(-\infty,0\right)\cup\left(t_{\pone,\pone}^*,\infty\right)$ and positive on $\left(0,t_{\pone,\pone}^*\right)$.
\item\label{lemma-boundarycase-2}If $\pone\in\left[1/2,p^+\right]$, then $f_{\pone,1-\pone}$ is symmetrical, has a root at $t=0$, and is negative on $\R\backslash\{0\}$.
\item\label{lemma-boundarycase-3}If $\pone>p^+$, then $f_{\pone,1-\pone}$ is symmetrical and has a root at $t=0$. Moreover, there exists $t_{\pone}^\dagger>0$ such that $f_{\pone,1-\pone}\left(\pm t_{\pone}^\dagger\right)$ vanishes and $f_{\pone,1-\pone}$ is negative on $\left(-\infty,-t_{\pone}^\dagger\right)\cup\left(t_{\pone}^\dagger,\infty\right)$ and positive on $\left(-t_{\pone}^\dagger,t_{\pone}^\dagger\right)\backslash\{0\}$.
\end{enumerate}
In particular, the function $g_{\pone,(\pone,1-\pone)^-}$ is strictly unimodal if and only if $\pone\leq p^+$.
\end{lemma}
\begin{proof}
The first statement follows from \cref{theorem-KS-orig} and the observations that $g_{p,p}=2g_p$ and $t_{p,p}^{\ast}=t_p^{\ast}$. For assertion \labelcref{lemma-boundarycase-2} we observe that $f_{\pone,1-\pone}$ is manifestly symmetrical and that its second derivative is, up to positive factors, a polynomial of degree two, namely
\begin{equation*}
f_{\pone,1-\pone}''(t)=\frac{\pone(1-\pone)t\ee^t\left(\ee^{2t}-1\right)}{\left[1+\pone\left(\ee^t-1\right)\right]^3\left[1+(1-\pone)\left(\ee^t-1\right)\right]^3}\mathfrak{p}_{\pone}\left(\ee^t\right),
\end{equation*}
where $\mathfrak{p}_{\pone}(x)=-\pone(1-\pone)\left[1-2\pone(1-\pone)\right](1+x^2)+\left[1-4\pone(1-2\pone^2+\pone^3)\right]x$. Since the discriminant of this polynomial is equal to $(1-2\pone)^2[1+2\pone(1-\pone)][1-6\pone(1-\pone)]$, it does not have any real roots if $1/2\leq\pone<p^+$. A quick computation shows that $f_{\pone,1-\pone}$ itself as well as its first three derivatives vanish at $t=0$, and that $f_{\pone,1-\pone}^{(4)}=4\pone(1-\pone)[1-6\pone(1-\pone)]$; this implies the claim for $\pone$ strictly less than $p^+$. For $\pone=p^+$ one checks that the only root of $\mathfrak{p}_{\pone}$ is equal to zero, which does not correspond to a zero of $\mathfrak{p}_{p^+}\left(\ee^t\right)$. Moreover, the first-non-vanishing derivative of $f_{p^+,1-p+}$ at zero (the sixth!) is equal to $-8/9$, proving the claim. We now turn to the proof of \labelcref{lemma-boundarycase-3}. If $\pone$ exceeds $p^+$, then the fourth derivative of $f_{\pone,1-\pone}$ at $t=0$ is positive, and the two roots $x_{\pone}^\pm$ of $\mathfrak{p}_{\pone}$ correspond to the abscissas $t_{\pone}^\pm=\log x_{\pone}^\pm$ of the only two inflection points of $f_{\pone,1-\pone}$. To complete the proof it only remains to check that $f_{\pone,1-\pone}^{(3)}\left(t_{\pone}^\pm\right)$ is positive and negative, respectively; the lengthy details are omitted.
\end{proof}

We now return to the general case and denote by $\Delta$ the triangle $\left\{\langle\pone,\ptwo\rangle\in[0,1]^2:\ptwo < (\pone,1-\pone)^-\right\}$. As in the generalised Bernoulli case, a large part of the proof of \cref{thm-unimodalitynew} hinges on the convexity properties of the function $f_{\pone,\ptwo}$. For the statement of the next result we introduce the notation
\begin{equation}
\label{eq-discrimant}
\mathcal{D}(\pone,\ptwo)=12\pone \ptwo \left(\pone \ptwo  + \pone  + \ptwo \right) - 14\pone \ptwo  + \pone ^2+\ptwo ^2,
\end{equation}
and
\begin{equation}
\label{eq-condA}
\mathcal{A}(\pone,\ptwo)=2\pone\ptwo-\pone(2\pone-1)-\ptwo(2\ptwo-1),
\end{equation}
which will also be used later on. We will also use the fact that the inclusion
\begin{equation}
\label{eq-inclusionDA}
D\subset A,\quad D=\left\{\mathcal{D}(\pone,\ptwo)\leq0\right\},\quad A=\left\{\mathcal{A}(\pone,\ptwo)\geq0\right\},
\end{equation}
holds, which follows from simple algebra.
\begin{lemma}
\label{lemma-inflectionpoints}
For every $\langle\pone,\ptwo\rangle\in\Delta$, the function $f_{\pone,\ptwo}$ has two inflection points with abscissas $t=0$ and $t=t_{\pone,\ptwo}^{\ast}/2$.
\begin{enumerate}
\item\label{lemma-inflectionpoints-1} If $\mathcal{D}(\pone,\ptwo)$ is negative or zero, then these are the only inflection points and $f_{\pone,\ptwo}$ is concave on $\left(-\infty,0\right)\cup\left(t_{\pone,\ptwo}^{\ast}/2,+\infty\right)$ and convex on $\left(0,t_{\pone,\ptwo}^{\ast}/2\right)$.
\item\label{lemma-inflectionpoints-2} If $\mathcal{D}(\pone,\ptwo)$ is positive and $\mathcal{A}(\pone,\ptwo)$ is zero, then the point $\left(t_{\pone,\ptwo}^*,0\right)$ is an additional inflection point of $f_{\pone,\ptwo}$ and there are no others.
\item\label{lemma-inflectionpoints-3} if $\mathcal{D}(\pone,\ptwo)$ is positive and $\mathcal{A}(\pone,\ptwo)$ is non-zero, then there exist two additional inflection points with abscissas $t_{\pone,\ptwo}^{\pm}$.
\begin{enumerate}
\item\label{lemma-inflectionpoints-3a} if $\mathcal{A}(\pone,\ptwo)$ is negative, then $t_{\pone,\ptwo}^-< 0$ and $t_{\pone,\ptwo}^+> t_{\pone,\ptwo}^{\ast}$, and $f_{\pone,\ptwo}$ is concave on $\left(-\infty,t_{\pone,\ptwo}^-\right)\cup\left(0,t_{\pone,\ptwo}^{\ast}/2\right)\cup\left(t_{\pone,\ptwo}^+,+\infty\right)$ and convex on $\left(t_{\pone,\ptwo}^-,0\right)\cup\left(t_{\pone,\ptwo}^{\ast}/2,t_{\pone,\ptwo}^+\right)$.
\item\label{lemma-inflectionpoints-3b} if $\mathcal{A}(\pone,\ptwo)$ is positive, then $0<t_{\pone,\ptwo}^-<t_{\pone,\ptwo}^{\ast}/2$ and $t_{\pone,\ptwo}^{\ast}/2<t_{\pone,\ptwo}^+<t_{\pone,\ptwo}^{\ast}$, and $f_{\pone,\ptwo}$ is concave on $\left(-\infty,0\right)\cup\left(t_{\pone,\ptwo}^-,t_{\pone,\ptwo}^{\ast}/2\right)\cup\left(t_{\pone,\ptwo}^+,+\infty\right)$ and convex on $\left(0,t_{\pone,\ptwo}^-\right)\cup\left(t_{\pone,\ptwo}^{\ast}/2,t_{\pone,\ptwo}^+\right)$.
\end{enumerate}
\end{enumerate}
\end{lemma}
\begin{proof}
Direct calculation shows that the second derivative of $f_{\pone,\ptwo}$ vanishes for $t=0$ and $t=t_{\pone,\ptwo}^{\ast}/2$ and that $f_{\pone,\ptwo}''(t)$ itself can be written as
\begin{equation*}
f_{\pone,\ptwo}''(t)=\frac{t \ee^t \left[\pone +\ptwo -1+\pone \ptwo \left(\ee^{2t}-1\right)\right]}{\left[1+\pone \left(\ee^t -1\right)\right]^3\left[1+\ptwo \left(\ee^t - 1\right)\right]^3}\mathfrak{p}_{\pone,\ptwo}\left(\ee^t\right),
\end{equation*}
where
\begin{align*}
\mathfrak{p}_{\pone,\ptwo}(x)=&(1-\pone)(1-\ptwo)(2\pone\ptwo-\pone-\ptwo)\\
  &+ \left[4\pone\ptwo(\pone+\ptwo-\pone\ptwo)-6\pone\ptwo+\pone^2+\ptwo^2\right]x\\
  &+ \pone\ptwo\left[2\pone\ptwo-\pone-\ptwo\right]x^2
\end{align*}
is a polynomial of degree two. For the existence of at least one additional inflection point it is thus necessary that the discriminant $\mathcal{D}(\pone,\ptwo)$ of $\mathfrak{p}_{\pone,\ptwo}$, which is given by \cref{eq-discrimant}, is non-negative. If the discriminant is zero, however, the only real root of the polynomial $\mathfrak{p}_{\pone,\ptwo}$ equals $\exp\left\{t_{\pone,\ptwo}^*/2\right\}$, which we we have found before, and so no additional inflection point exists in this case; this proves part \labelcref{lemma-inflectionpoints-1}.

If $\mathcal{D}(\pone,\ptwo)$ is positive, the roots of $\mathfrak{p}_{\pone\ptwo}$ are given by
\begin{equation*}
x_{\pone,\ptwo}^{\pm} = \frac{4\pone \ptwo \left(\pone +\ptwo -\pone \ptwo \right) -6\pone \ptwo  +\pone ^2+\ptwo ^2 \pm (\pone -\ptwo)\sqrt{\mathcal{D}(\pone,\ptwo)}}{2\pone \ptwo \left(\pone +\ptwo -2\pone \ptwo \right)}.
\end{equation*}
It is not difficult to check that $x_{\pone,\ptwo}^-$, and hence $x_{\pone,\ptwo}^+>x_{\pone,\ptwo}^-$, are never negative; the roots $x_{\pone,\ptwo}^\pm$ thus correspond to roots $t_{\pone,\ptwo}^{\pm}=\log x_{\pone,\ptwo}^{\pm}$ of $\mathfrak{p}_{\pone,\ptwo}\left(\ee^t\right)$ and thus to potential inflection points of $f_{\pone,\ptwo}$. To obtain a complete picture of the convexity properties of $f_{\pone,\ptwo}$ we will next analyse where the inflection points with abscissas $t_{\pone,\ptwo}^{\pm}$ are located relative to $t=0$ and $t=t_{\pone,\ptwo}^{\ast}/2$. After some more algebra we obtain that $t_{\pone,\ptwo}^-$ and $t_{\pone,\ptwo}^+$ lie to either side of $t_{\pone,\ptwo}^{\ast}/2$. We further find that $\mathcal{A}(\pone,\ptwo)<0$ is equivalent to $t_{\pone,\ptwo}^-$ being negative and $t_{\pone,\ptwo}^+$ exceeding $t_{\pone,\ptwo}^{\ast}$ (proving \labelcref{lemma-inflectionpoints-3a}), that $\mathcal{A}(\pone,\ptwo)>0$ is equivalent to $t_{\pone,\ptwo}^-$ being positive and $t_{\pone,\ptwo}^+$ deceeding $t_{\pone,\ptwo}^{\ast}$ (proving \labelcref{lemma-inflectionpoints-3b}), and finally that $\mathcal{A}(\pone,\ptwo)=0$ implies $t_{\pone,\ptwo}^-=0$ and $t_{\pone,\ptwo}^+=t_{\pone,\ptwo}^{\ast}$ (proving \labelcref{lemma-inflectionpoints-2}).

To complete the proof it remains to check the sign of the third derivative of $f_{\pone,\ptwo}$ at the (maximal) four inflection points; details of these straightforward computations are again omitted.
\end{proof}
Typical graphs of $t\mapsto f_{\pone,\ptwo}(t)$ for different values of $\langle\pone,\ptwo\rangle$, illustrating the convexity properties obtained in \cref{lemma-inflectionpoints}, are depicted in \cref{fig-seq}. In particular, part \labelcref{lemma-inflectionpoints-1} is illustrated in \cref{fig-seq-d,fig-seq-inf}; part \labelcref{lemma-inflectionpoints-2} in \cref{fig-seq-a}; part \labelcref{lemma-inflectionpoints-3a} in \cref{fig-seq-0}; and part \labelcref{lemma-inflectionpoints-3b} in \cref{fig-seq-ab,fig-seq-b,fig-seq-bc,fig-seq-c,fig-seq-cd}.

As in the generalised Bernoulli case we will establish necessary and sufficient conditions for $g_{\pone,\ptwo}$ to be unimodal by analysing the sign of the function $f_{\pone,\ptwo}$ on the three intervals $\left(-\infty,0\right)$, $\left(0,t_{\pone,\ptwo}^{\ast}\right)$ and $\left(t_{\pone,\ptwo}^{\ast},\infty\right)$, which is done in \cref{prop-step1new,prop-step2new,prop-step3new} below.
\begin{proposition}
\label{prop-step1new}
For every $\langle\pone,\ptwo\rangle\in\Delta$ the following are equivalent:
\begin{enumerate}
 \item\label{prop-step1new-1} the derivative of the function $g_{\pone,\ptwo}$ is positive on $\R^-$;
 \item\label{prop-step1new-2} the third derivative of $f_{\pone,\ptwo}$ is non-negative at $t=0$;
 \item\label{prop-step1new-3} $\pone<p^+$ and $\ptwo\geq\alpha(\pone)$, where $\alpha:[0,p^+]\to[0,p^-]$ is defined by
\begin{equation}
\label{eq-derivzero}
\alpha(\pone)=\max{\left\{0,\frac{1}{4}\left[1+2\pone -\sqrt{1+12p(1-p)}\right]\right\}}.
\end{equation}
\end{enumerate}
\end{proposition}
\begin{proof}
In order to show that \labelcref{prop-step1new-1} implies \labelcref{prop-step1new-2}, we observe that, for $t<0$, $g_{\pone,\ptwo}'(t)$ being positive is equivalent to $f_{\pone,\ptwo}(t)$ being negative. Since all three of $f_{\pone,\ptwo}(0)$, $f_{\pone,\ptwo}'(0)$ and $f_{\pone,\ptwo}''(0)$ are equal to zero, this implies that $f_{\pone,\ptwo}^{(3)}(0)$ is non-negative. The equivalence between \labelcref{prop-step1new-2} and \labelcref{prop-step1new-3} follows from the fact that
\begin{equation*}
f_{\pone,\ptwo}^{(3)}(0) = \pone(1-\pone)(1-2\pone)+\ptwo(1-\ptwo)(1-2\ptwo) = (1-\pone-\ptwo)\mathcal{A}(\pone,\ptwo)
\end{equation*}
and an easy computation. Finally, for $\langle\pone,\ptwo\rangle$ such that $\ptwo\geq\alpha(\pone)$ and $\mathcal{A}(\pone,\ptwo)$ thus non-negative, \Cref{lemma-inflectionpoints} shows that the function $f_{\pone,\ptwo}$ is concave on $\R^-$. In conjunction with the fact that $f_{\pone,\ptwo}(0)=f_{\pone,\ptwo}'(0)=0$, implies that $f_{\pone,\ptwo}(t)$ is negative for $t<0$, and thus proves \labelcref{prop-step1new-1}.
\end{proof}
The condition $\ptwo\geq\alpha(\pone)$, together with analogous inequalities for $\pone +\ptwo>1$ and $\pone<\ptwo$, is an alternative characterisation of the set $A$ from \cref{eq-inclusionDA} that features in the statement of \cref{thm-unimodalitynew}. Since the function $\alpha$ is convex and $\alpha'\left(p^+\right)$ is equal to one, the set $A$ is convex.

The next result gives a necessary and sufficient condition for the derivative of $g_{\pone,\ptwo}$ to be negative on the interval $\left(t_{\pone,\ptwo}^{\ast},\infty\right)$.
\begin{proposition}
\label{prop-step2new}
For every $\langle\pone,\ptwo\rangle\in\Delta$, the following are equivalent:
\begin{enumerate}
\item\label{prop-step2new-1}the derivative of the function $g_{\pone,\ptwo}$ is negative for $t>t_{\pone,\ptwo}^{\ast}$;
\item\label{prop-step2new-2}the first derivative of $f_{\pone,\ptwo}$ is non-positive at $t=t_{\pone,\ptwo}^{\ast}$;
\item\label{prop-step2new-3}$\pone<p^+$ and $\ptwo\geq\beta(\pone)$, where $\alpha(\pone)<\beta(\pone)<(\pone,1-\pone)^-$ is the unique solution to
\begin{equation}
\label{eq-derivtstar}
\log\left[\frac{1-\pone}{\pone}\frac{1-\beta(\pone)}{\beta(\pone)}\right]=\frac{2 (1-\pone -\beta(\pone))}{\pone (1-\pone)+\beta(\pone)(1-\beta(\pone))}.
\end{equation}
\end{enumerate}
\end{proposition}
\begin{proof}
We recall that the negativity of $g_{\pone,\ptwo}'$ on $\left(t_{\pone,\ptwo}^{\ast},\infty\right)$ is equivalent to the negativity of $f_{\pone,\ptwo}$ on that interval. As in the proof of \cref{prop-step1new} we thus see that \labelcref{prop-step2new-1} implies \labelcref{prop-step2new-2} because $f_{\pone,\ptwo}$ vanishes at $t=t_{\pone,\ptwo}^{\ast}$. We next prove the equivalence of \labelcref{prop-step2new-2} and \labelcref{prop-step2new-3}. Direct calculation informs us that 
\begin{align*}
h_{\pone}(\ptwo)\coloneqq& \frac{f_{\pone,\ptwo}'\left(t_{\pone,\ptwo}^{\ast}\right)}{\pone(1-\pone)+\ptwo(1-\ptwo)}\\
  =&\log\left[\frac{1-\pone}{\pone}\frac{1-\ptwo}{\ptwo}\right] - \frac{2(1-\pone-\ptwo)}{\pone(1-\pone)+\ptwo(1-\ptwo)}\\
\intertext{and}
\frac{\dd}{\dd \ptwo}h_{\pone}(\ptwo) =& -\frac{(1-\pone-\ptwo)^2}{\ptwo(1-\ptwo)\left[\pone(1-\pone)+\ptwo(1-\ptwo)\right]^2}\mathfrak{q}_{\pone}(\ptwo),
\end{align*}
where $\mathfrak{q}_{\pone}(x)=\pone^2-2(1+\pone)\ptwo+3\ptwo^2$ is a polynomial of degree two. We will show that, for each $\pone\in\left(0,p^+\right)$, the function $h_{\pone}$ has a unique root $\beta(\pone)<(\pone,1-\pone)^-$ and that $h_{\pone}(\ptwo)$ is positive for $\ptwo<\beta(\pone)$ and negative for $\ptwo>\beta(\pone)$. To this end we first observe that $\lim_{\ptwo\to0}h_{\pone}(\ptwo)=+\infty$ and $\lim_{\ptwo\to0}h_{\pone}'(\ptwo)=-\infty$ and that the derivative $h_{\pone}'$ is, up to positive factors, a quadratic polynomial in $\ptwo$. For $\pone<1/2$, the upper boundary of $\Delta$ is given by $\ptwo=\pone$ and we thus compute
\begin{equation*}
h_{\pone}(\pone) = 2\log\left[\frac{1-\pone}{\pone}\right]-\frac{1-2\pone}{\pone(1-\pone)},\quad h_{\pone}'(\pone) = \frac{(1-2\pone)^2}{2\pone^2(1-\pone)^2}.
\end{equation*}
The former expression is negative because it vanishes for $\pone=1/2$ and has a positive $\pone$-derivative equal to $(1-2\pone)^2/[\pone(1-\pone)]^2$; the latter expression is manifestly positive. The existence of a unique root $\beta(\pone)$ with the claimed property thus follows from the intermediate value theorem. For $\pone>1/2$ the upper boundary is given by $\ptwo=1-\pone$; in this case we compute
\begin{equation*}
h_{\pone}(1-\pone) =  h_{\pone}'(1-\pone) =  h_{\pone}''(1-\pone) = 0, \quad h_{\pone}^{(3)}(1-\pone) = -\frac{1-6\pone(1-\pone)}{2\left[\pone(1-\pone)\right]^3}.
\end{equation*}
If $\pone$ is less than $p^+$, the last expression is positive and the intermediate value theorem guarantees the existence of a unique root $\beta(\pone)$ as before. If $\pone$ exceeds $p^+$, there is no such root. To see this, it is enough to compute the smallest stationary point of the function $\ptwo\mapsto \mathfrak{q}_{\pone}(\ptwo)$, which is given by $x=\left(1+\pone+\sqrt{2\pone(1-\pone)+1}\right)/3$. and observe that it exceeds $1-\pone$ if and only if $\pone$ is greater than $p^+$. The claim that $\beta(\pone)$ exceeds $\alpha(\pone)$ follows from the fact that $h_{\pone}(\alpha(\pone))$ is positive for $0<\pone<p^+$, which is a tedious, but not difficult, calculation, the details of which we omit.

Lastly, we prove the implication $\labelcref{prop-step2new-2}\wedge\labelcref{prop-step2new-3}\Rightarrow\labelcref{prop-step2new-1}$: if $\ptwo\geq \beta(\pone)>\alpha(\pone)$, then \cref{lemma-inflectionpoints} implies that the function $f_{\pone,\ptwo}$ is concave on $\left(t_{\pone,\ptwo}^{\ast},\infty\right)$, Thus the negativity of the first derivative $f_{\pone,\ptwo}'\left(t_{\pone,\ptwo}^{\ast}\right)$, in conjunction with the fact that the value $f_{\pone,\ptwo}\left(t_{\pone,\ptwo}^{\ast}\right)$ is zero, implies that $f_{\pone,\ptwo}(t)$, and hence $g_{\pone,\ptwo}'(t)$, are negative for $t$ exceeding $t_{\pone,\ptwo}^{\ast}$. 
\end{proof}
One checks directly that an explicit parametrisation of the graph of $\beta$ is given by
\begin{equation}
\label{eq-parambeta}
\mathfrak{b}:\mathbb{R}^+\to\Delta;\quad \tau\mapsto j(\tau)\left(\begin{array}{c}1\\1\end{array}\right)+\sqrt{-\partial_\tau j(\tau)}\left(\begin{array}{c}1\\-1\end{array}\right);\quad j(\tau)=\frac{1}{\tau}-\frac{1}{\ee^{\tau}-1},
\end{equation}and that $\mathfrak{b}(t)$ is a saddle point of the function $\langle \pone,\ptwo\rangle\mapsto f_{\pone,\ptwo}(t)$. In particular, letting $\tau\to0$, we find that $\beta\left(p^+\right)=p^-$. The parametrisation can also be used to show that $\beta$ is monotonely increasing and convex. The inequality $\ptwo\geq\beta(\pone)$, together with analogous inequalities for $\pone +\ptwo>1$ and $\pone<\ptwo$ define the set $B$ from the statement of \cref{thm-unimodalitynew}. Equivalently,
\begin{equation}
\label{eq-defB}
B=\left\{\langle\pone,\ptwo\rangle\in[0,1]^2:\left|\log\left(\frac{1-\pone}{\pone}\frac{1-\ptwo}{\ptwo}\right)\right|\leq\frac{2 |1-\pone -\ptwo |}{\pone (1-\pone)+\ptwo (1-\ptwo)}\right\}.
\end{equation}
The fact that $\beta(\pone)$ exceeds $\alpha(\pone)$ for all $\pone\in\left(0,p^+\right)$ translates directly into the inclusion $B\subset A$. The set $B$ is convex because the function $\beta$ is convex and $\beta'\left(p^+\right)=1$.

It remains to analyse the interval $\left(0,t_{\pone,\ptwo}^{\ast}\right)$. Before we prove, in \cref{prop-step3new}, that there exists a function $\gamma$ such that $f_{\pone,\ptwo}$ is positive on that interval if and only if $\ptwo>\gamma(\pone)$, we give a proof of \cref{thm-sufficientunimodality}.
\begin{proof}[Proof of \cref{thm-sufficientunimodality}]
The intersection of the triangular region $\Delta$ with the set $D$ defined in \cref{eq-inclusionDA} can be described as $\left\{\langle\pone,\ptwo\rangle\in[0,1]^2:\pone< p^+ \wedge \ptwo\geq\delta(\pone)\right\}$, where $\delta:\left[0,p^+\right]\to\left[0,p^-\right]$ is defined by
\begin{equation}
\label{eq-defdelta}
\delta(\pone) = \frac{\left(7-4\sqrt 3\right)\pone-\left(6-\sqrt 3\right)\pone^2}{1+12\pone(1-\pone)}.
\end{equation}
Similarly to $\alpha$ and $\beta$ the function $\delta$ is convex and satisfies $\delta'\left(p^+\right)=1$ which implies that the set $D$ is convex. The inclusion $D\subset B$ is thus equivalent to $\delta(\pone)>\beta(\pone)$; as in the proof of \cref{prop-step2new} this can be shown by verifying that $h_{\pone}(\delta(\pone))$ is negative for all $\pone\in\left(0,p^+\right)$. For $\langle\pone,\ptwo\rangle\in D\subset B\subset A$, the negativity of $g_{\pone,\ptwo}$ on $\left(-\infty,0\right)$ thus follows from \cref{prop-step1new}; the positivity of $g_{\pone,\ptwo}$ on $\left(t_{\pone,\ptwo}^{\ast},+\infty\right)$ follows from \cref{prop-step2new}. To see that $g_{\pone,\ptwo}$ is positive on the interval $\left(0,t_{\pone,\ptwo}^{\ast}\right)$ it suffices to recall that $f_{\pone,\ptwo}$ vanishes at $t=0$ and $t=t_{\pone,\ptwo}^{\ast}$, that the first non-zero derivative of $f_{\pone,\ptwo}$ at these points is positive and negative respectively, and that $f_{\pone,\ptwo}$ is convex in-between (\cref{lemma-inflectionpoints}).
\end{proof}

In the next lemma we analyse how the sign of $f_{\pone,\ptwo}(t)$ varies when $t$ and $\pone$ are held fixed and $\ptwo$ changes. We introduce the notations $p_t=1/[1+\exp(t/2)]$ and $\langle u(t),v(t)\rangle=\mathfrak{b}(t)$, defined in \cref{eq-parambeta}. The notation is illustrated in \cref{fig-Dg0}.

\begin{figure}
\centering
\includegraphics[width=1.\textwidth]{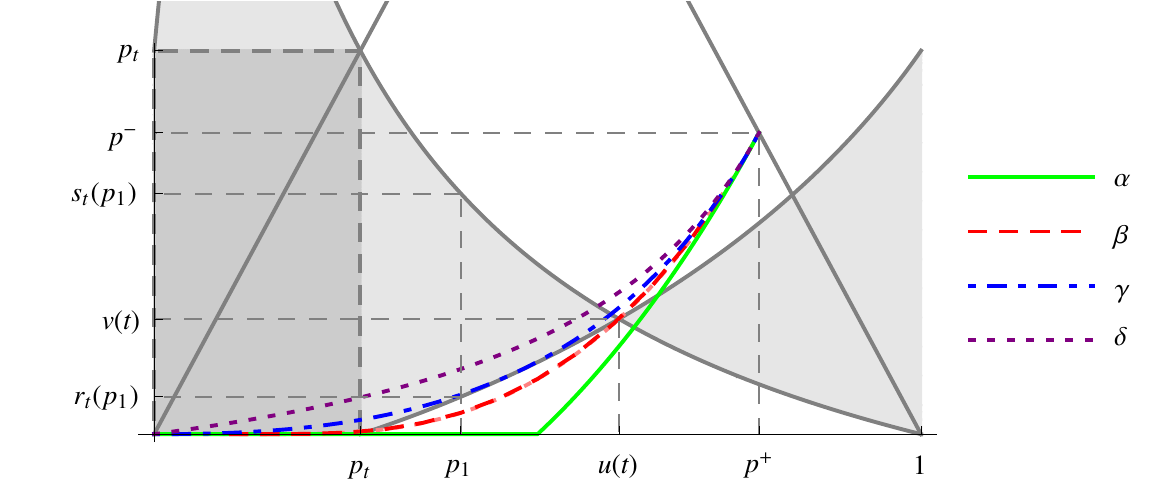}
\caption{Illustration of the notation used in \cref{lemma-Dgzero,prop-step3new}. 
The shaded region represents the set $\{f_{\pone,\ptwo}(t)>0\}$. 
The solid green curve is the graph of the function $\alpha$ defined in \cref{eq-derivzero} and represents the boundary for $g_{\pone,\ptwo}'(t)$ to be positive for $t<0$.
The red dashed curve, visualising the function $\beta$ defined parametrically in \cref{eq-parambeta}, represents the boundary for $g_{\pone,\ptwo}'(t)$ to be negative for $t>t_{\pone,\ptwo}^{\ast}$.
The blue dash-dotted curve represents the boundary for $g_{\pone,\ptwo}'(t)$ to be positive for $0<t<t_{\pone,\ptwo}^{\ast}$ and is derived from \cref{eq-secondordercondition}.
The purple line is the graph of the function $\delta$, defined in \cref{eq-defdelta}; it represents the boundary for $f_{\pone,\ptwo}$ to have exactly two inflection points.\\
The value of $\pone$ is $0.4$, the value of $t=2$.
}
\label{fig-Dg0}
\end{figure}

\begin{lemma}
\label{lemma-Dgzero}
For every $t>0$, the function $\langle\pone,\ptwo\rangle\mapsto f_{\pone,\ptwo}(t)$ has the the following properties:
\begin{enumerate}
\item\label{lemma-Dgzero-1} if $\pone<p_t$, then $f_{\pone,\ptwo}(t)$ is positive for all $\ptwo\in\left(0,(\pone,-1\pone)^-\right)$;
\item\label{lemma-Dgzero-2} if $p_t\leq \pone<u(t)$, then there exist $0<r_t(\pone)<s_t(\pone)<(\pone,1-\pone)^-$ such that $f_{\pone,\ptwo}(t)$ is negative for $\ptwo\in\left(0,r_t(\pone)\right)$; positive for $\ptwo\in\left(r_t(\pone),s_t(\pone)\right)$; and negative for $\ptwo\in\left(s_t(\pone),(\pone,1-\pone)^-\right)$. The larger boundary point $s_t(\pone)$ satisfies $t^{\ast}_{\pone,s_t(\pone)}=t$ which is equivalent to
\begin{equation}
\label{eq-stpone}
s_t(\pone)=\frac{1-\pone}{1+\pone \left(\ee^t-1\right)};
\end{equation}
\end{enumerate}
\end{lemma}
\begin{proof}
To establish \labelcref{lemma-Dgzero-1} we will prove the slightly stronger claim that $\langle\pone,\ptwo\rangle\mapsto f_{\pone,\ptwo}(t)$ is positive on the open square $\left(0,p_t\right)^2$. We first show that, if $\pone<p_t$, then $f_{\pone,\ptwo}(t)$ is positive for $\ptwo =0$ and $\ptwo =p_t$. To see that $f_{\pone,0}(t)$ is positive for $0<\pone<p_t$, we observe that it vanishes for $\pone =0$ and $\pone =p_t$, that
\begin{equation*}
\frac{\dd}{\dd \pone}f_{\pone,0}(t)=\frac{t\left(\ee^t-1\right)^2\pone ^2 - 2\left(\ee^t-1\right)\left(\ee^t-t-1\right)\pone  -2\left(\ee^t-1\right)+t\left(\ee^t+1\right)}{\left[1+\pone \left(\ee^t-1\right)\right]^2}
\end{equation*}
is, up to positive factors, a quadratic function in $\pone$, and that
\begin{equation*}
\left.\frac{\dd}{\dd \pone}f_{\pone,0}(t)\right|_{\pone =0} = \ee^t(t-2)+t+2 >0\;\text{and}\; \left.\frac{\dd}{\dd \pone}f_{\pone,0}(t)\right|_{\pone =p_t} = 2\left(t-2\sinh{(t/2)}\right) < 0.
\end{equation*}
For $\ptwo =p_t$ we obtain that $f_{\pone,p_t}(t)$ is positive for $\pone<p_t$ by the same argument, namely by noting that $f_{0,p_t}(t)=f_{p_t,p_t}(t)=0$, that $\partial_{\pone}f_{\pone,p_t}(t)$ is, up to positive factors, a quadratic function in $\pone$, and that the value of this derivative at the endpoints $\pone =0$ and $\pone =p_t$ is positive and negative, respectively.

Having established the positivity of $f_{\pone,\ptwo}(t)$ on two opposite sides of the square $\left(0,p_t\right)^2$ we can conclude the proof of the first part of the lemma by analysing the derivative in the direction orthogonal to these side, namely with respect to $\ptwo$. We find that
$\partial_{\ptwo} f_{\pone,\ptwo}(t)$ is up to positive factors a quadratic polynomial in $\pone$ which is positive for $\ptwo =0$ and negative for $\ptwo =p_t$. Therefore $f_{\pone,\ptwo}(t)$ is positive for all $0<\ptwo<p_t$, and in particular for $\ptwo<(\pone,1-\pone)^-<\pone<p_t$.

Part \labelcref{lemma-Dgzero-2} is proved in a similar way as part \labelcref{lemma-Dgzero-1}: we first show that $f_{\pone,0}(t)$ is negative: this follows from the observations that both $f_{p_t,0}(t)$ and $f_{1,0}(t)$ vanish, that $\partial_{\pone}f_{\pone,0}(t)$ has at most two roots in $[p_t,1]$, and that the values of this derivative at the endpoints $\pone =p_t$ and $\pone =1$ are negative and positive, respectively. Similarly, the facts that both $f_{p_t,1/2}(t)$ and $f_{1,1/2}(t)$ are negative, that $\partial_{\pone}f_{\pone,1/2}(t)$ has at most two roots in $[p_t,1]$, and that the values of this derivative at the endpoints $\pone =p_t$ and $\pone =1$ are negative and positive, respectively, show that  $f_{\pone,1/2}(t)$ is negative if $\pone>p_t$.

Turning now to the derivative with respect to $\ptwo$, we reiterate that the derivative of $f_{\pone,\ptwo}(t)$ with respect to $\ptwo$ has at most two roots in $[0,\pone ]$. We further note that
\begin{equation*}
\left.\frac{\dd}{\dd \ptwo}f_{\pone,\ptwo}(t)\right|_{\ptwo =0}>0,\quad\text{and}\quad\left.\frac{\dd}{\dd \ptwo}f_{\pone,\ptwo}(t)\right|_{\ptwo =1/2}=\frac{t\left[\cosh t +3\right] -4 \sinh t}{\cosh t+1}<0,
\end{equation*}
The fact that $f_{\pone,\ptwo}\left(t_{\pone,\ptwo}^{\ast}\right)$ vanishes implies that the function $\ptwo \mapsto f_{\pone,\ptwo}(t)$ has a zero at $\ptwo =s_t(\pone)<(\pone,1-\pone)^-\leq1/2$, and that $s_t(\pone)$ is given by \cref{eq-stpone}. If $p_t<\pone<u(t)$, the derivative $\partial_{\ptwo}f_{\pone,\ptwo}(t)$ is negative at $\ptwo =s_t(\pone)$ and hence, by the intermediate value theorem, there must exist another zero in the interval $\left(0,s_t(\pone)\right)$, which we call $r_t(\pone)$. There cannot be more than two zeros because $\partial_{\ptwo}f_{\pone,\ptwo}(t)$ is a quadratic function in $\ptwo$.
\end{proof}

\setlength{\fboxsep}{0.3pt}
\begin{figure}
       \centering
        \begin{subfigure}[b]{0.35\textwidth}
                \centering
                \fbox{\includegraphics[width=\textwidth]{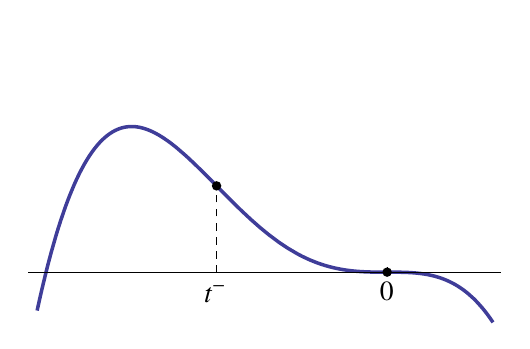}}
                \caption{$\ptwo=0$}
                \label{fig-seq-00}
        \end{subfigure}
        \hspace{1cm}
        \begin{subfigure}[b]{0.35\textwidth}
                \centering
                \fbox{\includegraphics[width=\textwidth]{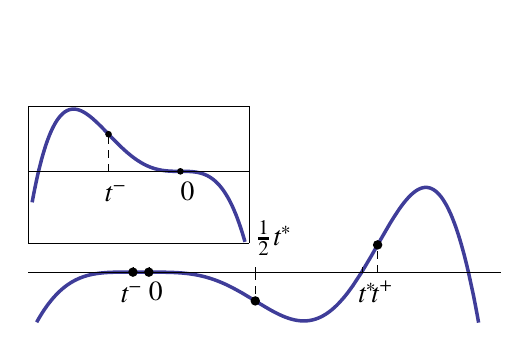}}
                \caption{$\ptwo<\alpha(\pone)$}
                \label{fig-seq-0}
        \end{subfigure}
        
       \vspace{.3cm}
       
	\begin{subfigure}[b]{0.35\textwidth}
                \centering
                \fbox{\includegraphics[width=\textwidth]{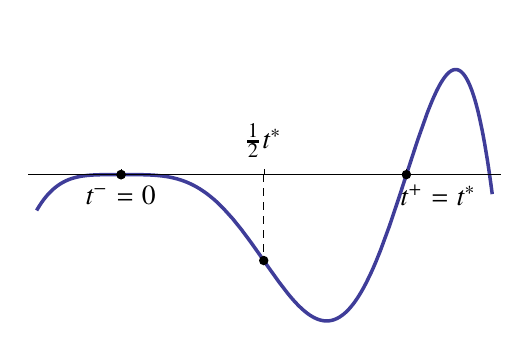}}
                \caption{$\ptwo=\alpha(\pone)$}
                \label{fig-seq-a}
        \end{subfigure}
        \hspace{1cm}
        \begin{subfigure}[b]{0.35\textwidth}
                \centering
                \fbox{\includegraphics[width=\textwidth]{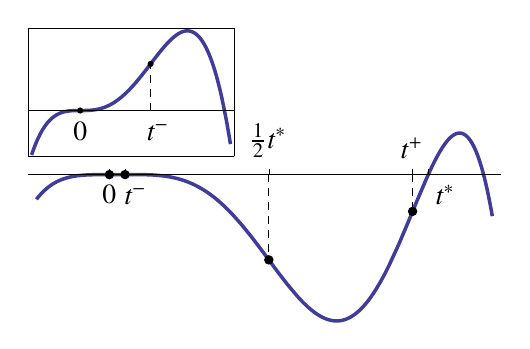}}
                \caption{$\alpha(\pone)<\ptwo<\beta(\pone)$}
                \label{fig-seq-ab}
        \end{subfigure}
        
       \vspace{.3cm}
        
	\begin{subfigure}[b]{0.35\textwidth}
                \centering
                \fbox{\includegraphics[width=\textwidth]{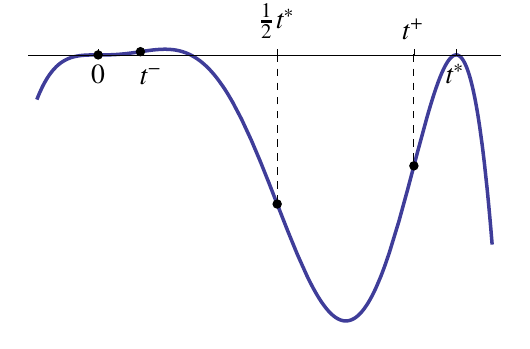}}
                \caption{$\ptwo=\beta(\pone)$}
                \label{fig-seq-b}
        \end{subfigure}
        \hspace{1cm}
        \begin{subfigure}[b]{0.35\textwidth}
                \centering
                \fbox{\includegraphics[width=\textwidth]{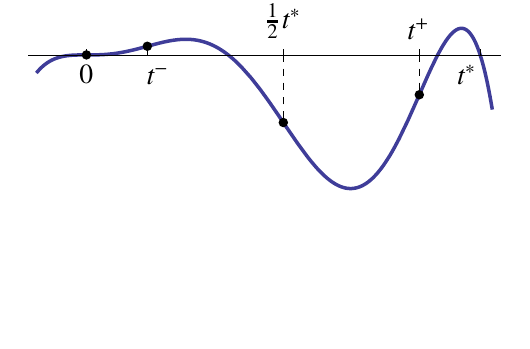}}
                \caption{$\beta(\pone)<\ptwo<\gamma(\pone)$}
                \label{fig-seq-bc}
        \end{subfigure}
	 
       \vspace{.3cm}
       
	\begin{subfigure}[b]{0.35\textwidth}
                \centering
                \fbox{\includegraphics[width=\textwidth]{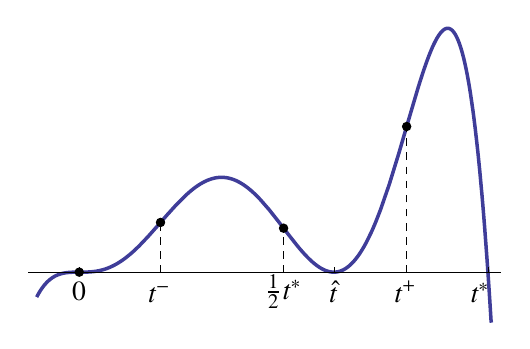}}
                \caption{$\ptwo=\gamma(\pone)$}
                \label{fig-seq-c}
        \end{subfigure}
        \hspace{1cm}
        \begin{subfigure}[b]{0.35\textwidth}
                \centering
                \fbox{\includegraphics[width=\textwidth]{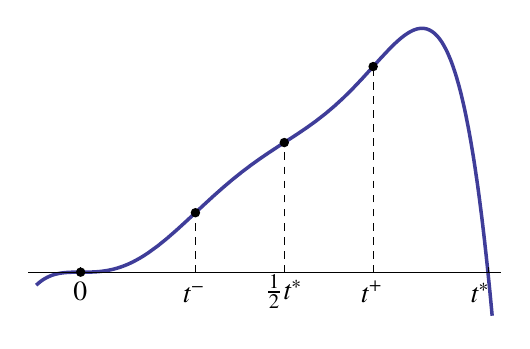}}
                \caption{$\gamma(\pone)<\ptwo<\delta(\pone)$}
                \label{fig-seq-cd}
        \end{subfigure}
	 
       \vspace{.3cm}
       
	\begin{subfigure}[b]{0.35\textwidth}
                \centering
                \fbox{\includegraphics[width=\textwidth]{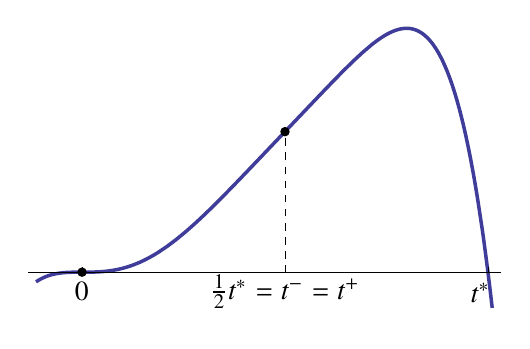}}
                \caption{$\ptwo=\delta(\pone)$}
                \label{fig-seq-d}
        \end{subfigure}
        \hspace{1cm}
        \begin{subfigure}[b]{0.35\textwidth}
                \centering
                \fbox{\includegraphics[width=\textwidth]{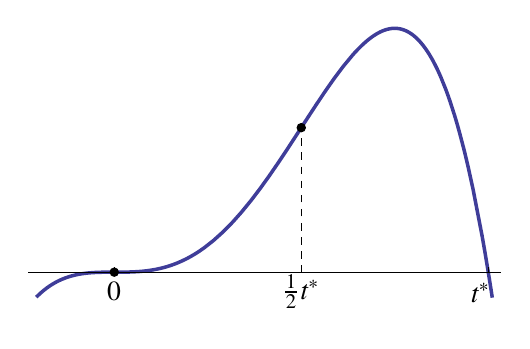}}
                \caption{$\ptwo>\delta(\pone)$}
                \label{fig-seq-inf}
        \end{subfigure}    
        \caption{Typical shapes of the functions $f_{\pone,\ptwo}$ as $\ptwo$ increases. The axes in the panels are not the same scale. We marked the location $t^{\ast}$ of the maximum of $g_{\pone,\ptwo}$ as well as the abscissas $0,t^-,t^+,t^{\ast}/2$ of the inflection points of $f_{\pone,\ptwo}$. The value of $\pone$ is $0.7$.}
        \label{fig-seq}
\end{figure}

\Cref{fig-Dg0} suggests that $\beta(\pone)$ exceeds $s_t(\pone)$ if and only $\pone\geq u(t)$. This is indeed true and can be proved by analysing the sign of $h_{\pone}\left(s_t(\pone)\right)$, where $h_{\pone}$ is the function defined in the proof of \cref{prop-step2new}. The calculations are lengthy, however, and therefore omitted. We now complete the proof of the unimodality of $g_{\pone,\ptwo}$ by analysing the sign of the function $f_{\pone,\ptwo}$ on the interval $\left(0,t_{\pone,\ptwo}^{\ast}\right)$.
\begin{proposition}
\label{prop-step3new}
There exists a function $\gamma:[0,p^+]\to[0,p^-]$ such that, for every $\langle\pone,\ptwo\rangle\in\Delta$, the derivative of the function $g_{\pone,\ptwo}$ is positive for $0<t<t_{\pone,\ptwo}^{\ast}$ if and only if $\pone<p^+$ and $\ptwo>\gamma(\pone)$.
\end{proposition}
\begin{proof}
We will show that for every $\pone\in\left(0,p^+\right)$ there exists a positive number $\gamma(\pone)\in\left[\beta(\pone),p^-\right)$ such that the function $f_{\pone,\ptwo}(t)$ is positive on the interval $\left(0,t_{\pone,\ptwo}^{\ast}\right)$ for all $\ptwo\in\left(\gamma(\pone),(\pone,1-\pone)^-\right)$.

Let $\pone<p^+$ be given. It follows from \cref{prop-step2new} that for $\ptwo<\beta(\pone)$, the function $f_{\pone,\ptwo}$ is not positive on $\left(0,t_{\pone,\ptwo}^{\ast}\right)$ because it vanishes at $t=t_{\pone,\ptwo}^*$ with positive derivative. It thus follows from the convexity properties of $f_{\pone,\ptwo}$ (\cref{lemma-inflectionpoints}) that for $\ptwo<\beta(\pone)$ the graph of $f_{\pone,\ptwo}$ looks as in \cref{fig-seq-ab}; in particular, it has a local minimum with abscissa in the interval $\left(t_{\pone,\ptwo}^{\ast}/2,t_{\pone,\ptwo}^{\ast}\right)$ and negative ordinate. As $\ptwo$ increases, first to $\beta(\pone)$ and then beyond, the derivative of $f_{\pone,\ptwo}$ at $t_{\pone,\ptwo}^*$ changes sign, as depicted in \cref{fig-seq-b,fig-seq-bc}. Increasing $\ptwo$ further, it follows from \cref{lemma-boundarycase} and the intermediate value theorem that there must exist a smallest $\gamma(\pone)$ in the interval $\left[\beta(\pone),(\pone,1-\pone)^-\right)$ such that $f_{\pone,\gamma(\pone)}$ has its local minimum at $\langle\hat t(\pone),0\rangle$ for some $\hat t(\pone)\in\left(0,t_{\pone,\gamma(\pone)}^{\ast}\right)$. The graph of $f_{\pone,\gamma(\pone)}$ is depicted in \cref{fig-seq-c}.

In the following we will show that $f_{\pone,\ptwo}(t)$ is positive on $\left(0,t_{\pone,\ptwo}^{\ast}\right)$ for all $\ptwo\in\left(\gamma(\pone),\pone\right)$. This is equivalent to showing that, for all positive $t$, the value $f_{\pone,\ptwo}(t)$ is positive for all points $\langle\pone,\ptwo\rangle\in\Delta$ such that $\ptwo>\gamma(\pone)$ and $t_{\pone,\ptwo}^{\ast}>t$. If $\pone<p_t$ this is true because \cref{lemma-Dgzero}, \labelcref{lemma-Dgzero-1} shows that $f_{\pone,\ptwo}$ is positive for all $\ptwo<\pone$. If  $p_t\leq\pone<u(t)$, we observe that, by \cref{eq-stpone}, the condition $t_{\pone,\ptwo}^{\ast}>t$ is equivalent to the condition $\ptwo<s_t(\pone)$ and that \cref{lemma-Dgzero}, \labelcref{lemma-Dgzero-2} shows that $f_{\pone,\ptwo}$ is positive for all $\ptwo\in\left(r_t(\pone),s_t(\pone)\right)$. It thus remains to show that $r_t(\pone)\leq\gamma(\pone)$, which follows from the fact that $f_{\pone,\gamma(\pone)}(t)$ is positive. Finally, if $\pone\geq u(t)$, then $\gamma(\pone)\geq\beta(\pone)\geq s_t(\pone)$, and there is thus nothing to prove.
\end{proof}

In the proof of \cref{prop-step3new}, the boundary $\gamma$ has been defined only implicitly. The argument showed, however, that the set $C$, which for $\ptwo<(\pone,1-\pone)^-$ is characterised by the condition $\ptwo>\gamma(\pone)$ and for other parts of $[0,1]^2$ by symmetry, is a non-empty subset of $B$. The inclusion $D\subset C$ follows from \cref{thm-sufficientunimodality} and the fact that \cref{prop-step3new} is an if-and-only-if statement. The proof also showed that the triple $\langle\pone,\gamma(\pone),\hat t(\pone)\rangle$ satisfies the equations
\begin{equation}
\label{eq-secondordercondition}
f_{\pone,\gamma(\pone)}\left(\hat t(\pone)\right)=f_{\pone,\gamma(\pone)}'\left(\hat t(\pone)\right)=0,\quad t_{\pone,\gamma(\pone)}^{\ast}/2<\gamma(\pone)<t_{\pone,\gamma(\pone)}^{\ast},
\end{equation}
and this can be used to solve for the boundary curve numerically. Alternatively, one might differentiate these equations implicitly and obtain a fairly complex system of differential equations for the functions $\gamma(\cdot)$, $\hat t(\cdot)$ which can be integrated numerically.

Finally, we can give a proof of our main result.
\begin{proof}[Proof of \cref{thm-unimodalitynew}]
For $\langle\pone,\ptwo\rangle\in\Delta$ parts \labelcref{thm-unimodalitynew1,thm-unimodalitynew2,thm-unimodalitynew3} are proved in \cref{prop-step1new,prop-step2new,prop-step3new}, respectively. The boundary case $\ptwo=(\pone,1-\pone)^-$ is treated in \cref{lemma-boundarycase}. For other values of $\pone$, $\ptwo$ the claim follows by symmetry.
\end{proof}


\end{document}